\documentclass{amsart}
\usepackage[utf8]{inputenc}

\usepackage[english]{babel}
\usepackage{amsthm, amsfonts, amssymb, amscd, fleqn}
\usepackage{nicefrac}
\usepackage{tikz}
\usepackage{enumerate}

\usetikzlibrary{arrows}
\usetikzlibrary{calc}
\usepackage{graphicx}
\usepackage[hyphens]{url}
\usepackage{babelbib}
\usepackage{etoolbox}

\newtheorem{theorem}{Theorem}[section]

\newtheorem{proposition}[theorem]{Proposition}
\newtheorem{corollary}[theorem]{Corollary}
\newtheorem{lemma}[theorem]{Lemma}

\theoremstyle{definition}
\newtheorem{remark}[theorem]{Remark}

\theoremstyle{definition}

\theoremstyle{remark}

\newcommand\restr[2]{{
  \left.\kern-\nulldelimiterspace 
  #1 
  \vphantom{\big|} 
  \right|_{#2} 
  }}
\newcommand{\vertiii}[1]{{\left\vert\kern-0.25ex\left\vert\kern-0.25ex\left\vert #1 
    \right\vert\kern-0.25ex\right\vert\kern-0.25ex\right\vert}}

\newcommand\norm[1]{\|#1\|}
\newcommand\vn[1]{|#1|^{\frac{1}{n}}}
\newcommand\vd[1]{|#1|^{\frac{1}{d^2}}}
\newcommand\vrn[2]{\left(\frac{|#1|}{|#2|}\right)^{\frac{1}{n}}}

\def\Sn{S^{n-1}}
\def\vr{\text{vr}}

\def\diam{\text{diam}}
\def\iso{\text{Iso}}

\def\absconv{\text{absconv}}

\def\x{\mathbb{x}}
\def\+{\mathbb{+}}

\def\R{\mathbb{R}}
\def\E{\mathbb{E}}

\def\N{\mathbb{N}}

\def\P{\mathbb{P}}
\def\PP{\mathcal{P}}

\def\BB{\mathcal{B}}
\def\MM{\mathcal{M}}

\def\NN{\mathcal{N}}
\def\OO{\mathcal{O}}

\def\PP{\mathcal{P}}

\def\SS{\mathcal{S}}

\def\LL{\mathcal{L}}

\def\lvr{\text{lvr}}
\newcommand{\Pro}{\mathbb P}

\def \lss{\preceq}
\def \gss{\succeq}
\def \tenes{\bigotimes_{\varepsilon_s}^{m,s} \ell_p^n}
\def \tenps{\bigotimes_{\pi_s}^{m,s} \ell_p^n}
\def \tene{\bigotimes_{\varepsilon}^{m} \ell_p^n}
\def \tenp{\bigotimes_{\pi}^{m} \ell_p^n}
\def \tenese{\bigotimes_{\varepsilon_s}^{m,s} E}
\def \tenpse{\bigotimes_{\pi_s}^{m,s} E}
\def \tenee{\bigotimes_{\varepsilon}^{m} E}
\def \tenpe{\bigotimes_{\pi}^{m} E}

\title[Largest volume ratio of a body]{Asymptotic estimates for the largest volume ratio of a convex body}

\author[D. Galicer]{Daniel Galicer}
\address{ Departamento de Matem\'{a}tica - IMAS-CONICET,
Facultad de Cs. Exactas y Naturales  Pab. I, Universidad de Buenos Aires
(1428) Buenos Aires, Argentina}
\email{dgalicer@dm.uba.ar}

\author[M. Merzbacher]{Mariano Merzbacher}
\address{ Departamento de Matem\'{a}tica - IMAS-CONICET,
Facultad de Cs. Exactas y Naturales  Pab. I, Universidad de Buenos Aires
(1428) Buenos Aires, Argentina}
\email{mmerzbacher@dm.uba.ar}

\author[D. Pinasco]{Dami\'an Pinasco}
\address{Departamento de Matem\'{a}ticas y Estad\'{\i}stica, Universidad T. Di Tella, Av. Figueroa Alcorta 7350 (1428), Buenos Aires, Argentina and CONICET}
\email{dpinasco@utdt.edu}



\thanks{The second author was supported by a CONICET doctoral fellowship.
This was partially supported by  CONICET PIP 11220130100329, CONICET PIP 11220090100624, ANPCyT PICT 2015-2299}

\begin{document}


\begin{abstract}
The
\emph{largest volume ratio} of given convex body $K \subset \R^n$  is defined as
$$\lvr(K):= \sup_{L \subset \R^n} \vr(K,L),$$
    where the $\sup$ runs over all the convex   bodies $L$.
   We prove the following sharp lower bound  $$c \sqrt{n}  \leq \lvr(K)  ,$$ for \emph{every} body $K$ (where $c>0$ is an absolute constant). This result improves the former best known lower bound, of order  $\sqrt{\frac{n}{\log \log(n)}}$.  
   
    We also study the exact asymptotic behavior of the largest volume ratio for some natural classes. In particular, we show that $\lvr(K)$ behaves as the square root of the dimension of the ambient space in the following cases: if $K$ is the unit ball of an unitary invariant norm in $\R^{d \times d}$ (e.g., the unit ball of the $p$-Schatten class $S_p^d$ for any $1 \leq p \leq \infty$), $K$ is the the unit ball of the full/symmetric tensor product of $\ell_p$-spaces endowed with the projective or injective norm or $K$ is unconditional. 
\end{abstract}


\maketitle
\section{Introduction}

For many applications in asymptotic geometric analysis, convex geometry or even optimization it is useful to approximate a given convex body by another one. For example, the classical Rogers-Shephard  inequality \cite[Theorem 1.5.2]{artstein2015asymptotic}  states that, for a convex body $K \subset \R^n$, the volume of the difference body $K-K$
is  ``comparable'' with the volume of $K$. Precisely, $\vn{K-K} \leq 4 \vn{K}$ where  $|\cdot|$ stands for the $n$-dimensional Lebesgue measure.  Rogers and Shephard also showed, with the additional assumption that $K$ has barycenter at the origin, that the intersection body $K \cap (-K)$ has ``large'' volume. Namely, $\vn{K \cap (-K)} \geq \frac{1}{2} \vn{K}$.
These inequalities imply that any given body is enclosed by (or contains) a symmetric body whose volume is ``small'' (``large'') enough.
In many cases this allows us to take advantage of the symmetry of the difference body (or the intersection body) to conclude something about $K$.

Another interesting example of Milman and Pajor \cite[Section 3]{milman1989isotropic} shows that \begin{align}\label{Milman-Pajor}
    L_K \leq c \inf \left\{ \left(\frac{|W|}{|K|}\right)^{\frac{1}{n}} :  W \mbox{ is unconditional and contains } K \right\},
    \end{align}
where $L_K$ stands for the isotropic constant of $K \subset \R^n$ (see  \cite[Section 2.3.1]{brazitikos2014geometry}) and $c>0$ is an absolute constant.
Therefore, having a good approximation of $K$ by an unconditional convex body provides structural geometric information of $K$.

Perhaps the most notable application of these kind of approximations can be viewed when studying John/L\"owner ellipsoid (maximum/minimum volume ellipsoid respectively). For example, if the Euclidean ball is the maximal volume ellipsoid inside $K$, we can decompose the identity as a linear combination of rank-one operators defined by contact points \cite[Theorem 2.1.10]{artstein2015asymptotic}. This decomposition plays a key role in the study of distances between bodies, see \cite{tomczak1989banach} for a complete treatment on this. We also refer to
\cite{matouvsek2002lectures,gruber2007convex,giannopoulos2001john,lassak1992banach, lassak1998approximation,pelczynski1983structural} for many nice results/applications which involve these extremal ellipsoids.
A natural quantity that relates a given body $K$ with its ellipsoid of maximal volume is given by the ``standard'' \emph{volume ratio}
\begin{align}\label{volume ratio clasico}
\vr(K) = \inf \left\{ \left(\frac{|K|}{|\mathcal E|}\right)^{\frac{1}{n}} : \mathcal E \mbox{ is an ellipsoid contained in } K \right\}.
\end{align}
Using the Brascamp-Lieb inequality,  Ball showed that $\vr(K)$ is maximal  when $K$ is a simplex . The extreme case, among all the centrally symmetric convex bodies, is given by the cube (see \cite[Theorem 2.4.8]{artstein2015asymptotic}).

A natural generalization of this ratio is given by the following definition introduced by Giannopoulos and Hartzoulaki \cite{giannopoulos2002volume} and also studied by Gordon, Litvak,  Meyer and Pajor \cite{gordon2004john}: given two convex bodies $K$ and $L$ in  $\R^n$  the \emph{volume ratio} of the pair $(K,L)$ is defined as
\begin{align} \label{volumeratiogeneralizado}
 \vr(K,L):= \inf \left\{ \left(\frac{|K|}{|T(L)|}\right)^{\frac{1}{n}} : T(L) \mbox{ is  contained in } K  \right\},
 \end{align}
where the infimum (actually a minimum) is taken over all affine transformations $T$.
In other words, $\vr(K,L)$ measures how well can $K$ be approximated by an affine image of $L$. Note that the classic value $\vr(K)$ is just $\vr(K,B_2^n)$ where $B_2^n$ is the Euclidean unit ball in $\R^n$.

Given a convex body $K$, it is natural to ask how ``good'' an approximation of this kind can be (in terms of the dimension of the ambient space). Namely, we want to known  \emph{how large the value  $\vr(K,L)$ is} (for arbitrary convex bodies $L \subset \R^n$). Thus, it is important to compute the \emph{largest volume ratio of $K$}, given by
$$\lvr(K):= \sup_{L \subset \R^n} \vr(K,L),$$
where the $\sup$ runs over all the convex   bodies $L$.

Khrabrov (based on the well-known construction due to Gluskin \cite{gluskin1981diameter} used to understand the diameter of Minkowski compactum), showed in \cite[Theorem 5]{khrabrov2001generalized} the following:

{\sl
 For any convex body $K$ in $\R^n$ there is another body  $L \subset \R^n$ such that
\begin{align} \label{khravrov}
    c \sqrt{\frac{n}{\log \log(n)}} \leq \vr(K,L)
\end{align}
where $c>0$ is an absolute constant.}
The body $L$ is found using the probabilistic method (Khrabrov considered a random polytope whose vertices are sampled on the unit sphere and showed that, with high probability, it verifies Equation~\eqref{khravrov}).

On the other hand, it is very easy to see that $\vr(K,L) \leq n$ for every pair $(K,L)$.
Using Chevet's inequality together with clever positions of $K$ and $L$, Giannopoulos and Hartzoulaki \cite{giannopoulos2002volume} were able to prove the following important and stronger result:

{\sl Let $K$ and $L$ be two convex bodies in $\R^n$. Then
\begin{align}\label{GH}
    \vr(K,L) \leq c \sqrt{n} \log(n),
\end{align}
where $c>0$ is an absolute constant.}

Combining the results of Khravrov and  Giannopoulos-Hartzoulaki, i.e.,  Equations  \eqref{khravrov} and \eqref{GH}, we get:

{\sl For any  convex body $K$ in $\R^n$, its largest volume ratio verifies

\begin{align*}
\sqrt{\frac{n}{\log \log(n)}} \ll    \lvr(K) \ll  \sqrt{n} \log(n).
\end{align*}
}
The well known result of John \cite[Theorem 2.1.3]{artstein2015asymptotic} asserts that for any convex body $L\subset \R^n$ we have $\vr(B_2^n,L) \ll \sqrt{n}$. It is not difficult to see that $\sqrt{n} \ll \vr(B_2^n,B_1^n),$ where  $B_1^n$ stands for the unit ball of $\ell_1^n$ thus,
$$\lvr(B_2^n) \sim \sqrt{n}.$$

In \cite[Theorem 1.3]{galicer2017minimal} the authors of this article showed that any convex body $L \subset \R^n$ can be inscribed in a simplex $S$ such that $\vn{S} \leq c \sqrt{n} \vn{L}$, where $c>0$ is an absolute constant. In other words, if $S$ is a simplex then $\vr(S,L) \ll \sqrt{n}$, for every convex body $L \subset \R^n$.
Since the regular simplex is the minimal volume simplex that contains the Euclidean unit ball (see \cite[Example 2.7]{galicer2017minimal}), by computing volumes  we have $\sqrt{n} \ll \vr(S,B_2^n)$.
Therefore, for a  simplex $S$, we know the exact asymptotic behaviour of its largest volume ratio:
$$\lvr(S) \sim \sqrt{n}.$$
Therefore the largest volume ratio of a convex body, in many cases, behaves as the square root of the dimension (of the ambient space).

We prove the following lower bound, that substantially improves  \eqref{khravrov}.

\begin{theorem} \label{teo principal}
For any convex body $K$ in $\R^n$ there is another body  $L \subset \R^n$ such that
\begin{align} \label{khravrov mejorado}
    c \sqrt{n} \leq \vr(K,L)
\end{align}
where $c>0$ is an absolute constant. In other words,
\begin{align} \label{khravrov mejorado lvr}
    \sqrt{n} \ll \lvr(K).
\end{align}

\end{theorem}
Moreover, we show that there are ``many'' (with high probability) random polytopes $L$ which verify equation \eqref{khravrov mejorado}. 
As we saw before in the previous examples, this lower bound cannot be improved in general. 

To obtain Theorem \ref{teo principal} we make some important changes in Khravrov's proof, which require finer estimates, and use some approximation arguments.

\medskip

We also deal, for same natural classes of convex bodies, with the upper bounds. Our results are of probabilistic nature, so we will be interested in obtaining bounds with high probability.

First we treat the case of the Schatten trace classes, the non-commutative version of the classical $\ell_p$ sequence spaces. They consist of all compact operators on a Hilbert space for which the sequence of their singular values belongs to $\ell_p$.
Many different properties of them in the finite dimensional setting have been largely studied in the area of asymptotic geometric analysis. For example, K\"oning, Meyer and Pajor \cite{konig1998isotropy} established the boundedness of the isotropic constants of the unit balls of $\SS_p^d \subset \R^{d \times d}$ ($1 \leq p \leq \infty$), Gu\'edon and Paouris \cite{guedon2007concentration} also studied concentration mass properties for the unit balls, Barthe and Cordero-Eurasquin \cite{barthe2013invariances} analyzed variance estimates, Radke and Vritsiou \cite{radke2016thin} proved the thin-shell conjecture, and recently Kabluchko, Prochno and Th\"ale \cite{kabluchko2018exact} exhibited the exact asymptotic behaviour of the volume and standard volume ratio; just to mention a few.

Therefore it is natural to try to understand what happens with the largest volume ratio of their unit ball. The following theorem provides an answer to this query.

\bigskip

\begin{theorem}
Let $1\leq p \leq \infty$ and $\SS_p^d \subset \R^{d \times d}$ be the $p$-Schatten class. The largest volume ratio of its unit ball, $B_{\SS_p^d}$, behaves as
\begin{align}
    \lvr(B_{\SS_p^d}) \sim d.
\end{align}
\end{theorem}

Moreover, we show that this also holds for the unit ball of any unitary invariant norm in $\R^{d \times d}$ (which follows from Theorem \ref{teo principal} and Corollary  \ref{cota superior unitary}  below).

Our approach is based on
Giannopoulos-Hartzoulaki's techniques.  We  show that if $L \subset \R^n$ is an arbitrary body then with ``high probability'' we can find  transformations $T$ such that $T(L) \subset \SS_p^d$ and  $$ \left(\frac{|\SS_p^d|}{|T(L)|}\right)^{\frac{1}{d^2}} \ll d.
$$

We also study the largest volume ratio for the unit ball of full or symmetric tensor products of $\ell_p$-spaces endowed with the well-known projective and injective tensor norms. Note that these spaces are identified with bounded/nuclear linear forms or homogeneous polynomials over $\ell_{p'}$-spaces. We obtain the following.

\begin{theorem}\label{lvr tensores}
Let $1\leq p \leq \infty$ and $X = \tene, \tenes, \tenp$ or $\tenps$. Then we have 
\begin{align}
\lvr(B_X) \sim \sqrt{\dim(X)},
\end{align}
where $\dim(X)$ stands for the dimension of $X$ as a vector space.
\end{theorem}

Recall that a convex body $K$ is called unconditional if for every choice of signs $(\varepsilon_k)_{k=1}^n \subset~\{-1,+1\}^n$, the vector $(\varepsilon_1 x_1, \dots, \varepsilon_{n} x_n)$ lies in $K$ if and only if $(x_1, \dots, x_n)$ is  in $K$. 
We also study the asymptotic behaviour of the largest volume ratio for unconditional bodies.
\begin{theorem}\label{teo incond}
Let $K \subset \R^n$ be an unconditional convex body.  Then,
\begin{align}
    \lvr(K) \sim \sqrt{n}.
\end{align}

\end{theorem}
The fact that $\lvr(K) \ll \sqrt{n}$ if $K$ is unconditional might be known for experts (although, as far as we know, is not explicitly stated elsewhere) and is a consequence of a mixture of the existence of Dvoretzky-Rogers' parallelepiped and a result of Bobkov-Nazarov.
As a result of a theorem of Pivovarov, we present a random version of Dvoretzky-Rogers' parallelepiped construction (which we believe is interesting in its own right) and show that, if $K \subset \R^n$ is unconditional and  $L \subset \R^n$ is an arbitrary body, then with ``high probability'' we can find  transformations $T$ such that $T(L) \subset K$ and  $$ \left(\frac{|K|}{|T(L)|}\right)^{\frac{1}{n}} \ll \frac{\sqrt{n}}{L_{L^\circ}},
$$
where $L_{L^{\circ}}$ stands for the isotropic constant of the polar body $L^{\circ}$.

\section{Preliminaries}

If $(a_{n})_{n}$ and $(b_{n})_{n}$ are two sequences of real numbers we write $a_{n} \ll b_{n}$ if there
exists an absolute constant $c>0$ (independent of $n$) such that $a_{n} \leq c b_{n}$ for every $n$.
We write $a_{n} \sim b_{n}$ if $a_{n} \ll b_{n}$ and $b_{n} \ll a_{n}$.
We denote by $e_1, \dots, e_n$ the canonical vector basis in $\R^n$ and by $S^{n-1}$, the unit sphere in $\R^n$.
We denote by $\absconv \{X_1, \dots, X_m \}$ the absolute convex hull of the vectors  $X_1, \dots, X_m$. That is,
$$\absconv\{X_1, \dots, X_m\} : = \left\{ \sum_{i=1}^m a_i X_i : \sum_{i=1}^m \vert a_i \vert \leq 1 \right\} \subset \R^n.$$

A convex body $K \subset \R^n$ is a compact convex set with non-empty interior. If $K$ is centrally symmetric (i.e., $K = -K$) we denote by $X_K$ the norm space $(\R^n, \Vert \cdot \Vert_{X_K})$ that has $K$ as its unit ball.

The polar set of $K$, denoted by $K^{\circ}$, is defined as
\begin{align*}
K^{\circ}= \{ x \in \R^n : \langle x, y \rangle \leq 1 \mbox{ for all } y\in K \}.
\end{align*}

The following result relates the volume of a body with the volume of it polar and is due to Blaschke-Santal\'o and Bourgain-Milman \cite[Theorem 1.5.10 and Theorem 8.2.2]{artstein2015asymptotic}:
{\sl
If $K$ is centrally symmetric then
\begin{align} \label{santalo}
\vn{K}\vn{K^{\circ}}  \sim \frac{1}{n}\cdot
\end{align}
}

A probability measure $\mu$ on $\R^n$ is isotropic if its center of mass is the origin
\begin{align*}
    \int_{\R ^n}\langle x,\theta\rangle \, d\mu(x)=0\mbox{ for every } \theta\in S^{n-1},
\end{align*}
and
\begin{align}
    \int_{\R ^n}\langle x,\theta\rangle^2 \, d\mu(x)=1 \mbox{ for every } \theta\in S^{n-1}.
\end{align}
A convex body $K \subset \R^n$ is said to be in isotropic position (or simply, is isotropic) if  it has volume one and its uniform measure is, up to an appropriate re-scaling, isotropic. In that case, its isotropic constant $L_K$, is given by
$$
L_K:= \left(\int_K x_1^2 \; dx\right)^{1/2}.
$$

Given a convex body $K$ in $\R^n$ with center of mass at the origin, there exists $A \in GL(n)$ such that $A(K)$ is isotropic \cite[Proposition 10.1.3]{artstein2015asymptotic}. Moreover, this isotropic image is unique up to orthogonal transformations; consequently, the isotropic constant $L_K$ results an invariant of the linear class of $K$. In some sense, the isotropic constant $L_K$ measures the spread of a convex
body $K$.

For a centrally symmetric convex body $K \subset \R^n$, its $\ell$-norm is defined as
$$\ell(K) := \int_{\R^n} \Vert (x_1, \dots, x_n) \Vert_{X_K} d\gamma(x),$$
where $d\gamma$ is the standard Gaussian probability in $\R^n$. For more information about this parameter see \cite[Chapter 12]{tomczak1989banach}. Recall that
$\ell(K) \sim \sqrt{n} w(K^{\circ}),$ where $w(\cdot)$ stands for the mean width (see \cite[Chapter 1.5.5]{artstein2015asymptotic}). 

Given a convex body $K$, $\iso(K)$ is the set of isometries of $K$, that is set of orthogonal transformations $O$ such that $O(K) = K$. We say that $K$ has enough symmetries if the only operator that commutes with every $T \in \iso (K)$ is the identity operator.
A convex body with enough symmetries is almost in John position \cite[Proposition 4.8]{aubrun2017alice}. More precisely, $\norm{id: \ell_2^n \to X_K}^{-1} B_2^n$ is the maximal volume ellipsoid contained in $K$. That means that if $K$ has enough symmetries then
\begin{align}\label{vr simetrias}
\vr(K)= \vr(K,B_2^n) \sim \norm{id:\ell_2^n \to X_K}\sqrt{n}|K|^{\frac{1}{n}}.
\end{align}

Natural examples of  convex bodies with enough symmetries are the unit balls of the Schatten classes. Given a matrix $T \subset \R^{d\times d}$ consider $s(T) = (s_1(T),\dots,s_d(T))$  the sequence of eigenvalues of $(TT^{*})^{\frac{1}{2}}$ (the singular values of $T$). The $p$-Schatten norm of $T \in \R^{d \times d}$ is defined as
\begin{align}
\sigma_p(T) = \norm{s(T)}_{\ell_p^d};
\end{align}
that is, the $\ell_p$-norm of the singular values of $T$. The $p$-Schatten norm arises as a generalization of the classical Hilbert-Schmidt norm. The analysis of the Schatten norm has a long tradition in local Banach space theory and their properties are widely studied. We denote by  $B_{\SS^d_p} \subset \R^{d \times d}$ the unit ball of $(\R^{d \times d}, \sigma_p)$.

\bigskip

We now review the basics definitions regarding tensor products.  We refer to \cite{defant1992tensor,floret1997natural} for a complete treatment on the subject. 
Given a normed space $E$ we write $\bigotimes^m E$ for the $m$-fold tensor product  of $E$, and $\bigotimes^{m,s} E$ for the symmetric $m$-fold tensor product, that is, the subspace of $\bigotimes^m E$ consisting of all tensor that can be written as $\sum_{i= 1} ^k \lambda_i \otimes^m x_i$, where $\lambda_i \in \R$ and $\otimes^m x_i = x_i \otimes \dots \otimes x_i$.
Observe that if $E$ has dimension $n$, $\dim(\bigotimes^m E) = n^m$ and $\dim(\bigotimes^{m,s} E) = \binom{m + n - 1}{n - 1}$. Since we consider $m$ as a fixed number, we have that in both cases the dimension of the space behaves like $n^m$ for $n$ large.

There are many norms than can be defined on the tensor product, we will focus on two of them. 
The projective tensor norm is defined as
\begin{align*}
\pi(x) = \inf\left\{\sum_{j=1}^r \prod_{i=1}^m\norm{x^r_i}_E  \right\},
\end{align*}
where the infimum is taken over all representations of $x$, $x = \sum_{i=1}^r x_1 \otimes \dots \otimes x_m$.
The injective tensor norm is defined as
\begin{align*}
\varepsilon(x) = \sup \left|\sum_{j=1}^r\prod_{i=1}^m|\varphi_i(x_i)| \right|,
\end{align*}
where the supremum runs over all $\varphi_1,\dots,\varphi_m \in E'$ and $\sum_{j=1}^r x_1 \otimes \dots \otimes x_m$ is a fixed representation of $x$.
Let $\alpha = \varepsilon$ or $\pi$, we write $\bigotimes_{\alpha} ^m$ for $m$-fold product endowed with the norm $\alpha$.

In the same way we can define the corresponding projective and injective norms in the symmetric setting. The symmetric projective norm is given by
\begin{align*}
\pi_s(x) :=  \inf\left\{\sum_{i=1}^r \norm{x_i}_E^m \right\},
\end{align*}
where the infimum is taken over all the representation of $x$ of the form $x = \sum_{i=1}^r \otimes^m x_i$.

The symmetric injective norm is computed as follows,
\begin{align*}
\varepsilon_s (x)= \sup_{\varphi \in B_{E'}} \left| \sum_{i=1}^r \varphi(x_i)^m \right|,
\end{align*} 
where $ x= \sum_{i=1}^r \otimes ^m x_i$ is a fixed representation of $x$.

\bigskip
We now recall some basic properties of the volume ratio defined in Equation~\eqref{volumeratiogeneralizado} which can easily be found in \cite{khrabrov2001generalized}.

\begin{remark} \label{propiedades elementales}
For every pair of centrally symmetric convex bodies $(K,L)$ in $\R^n.$ the following holds:
\begin{enumerate}
\item $$\vr(K,L) = \left(\frac{|K|}{|L|}\right)^{\frac{1}{n}} \cdot \inf_{T \in SL(n,\R)} \Vert T: X_L \to X_K \Vert,$$ where the infimum runs all over the linear transformations $T$ that lie on the special linear group of degree $n$ (matrices of determinant one).

\item $\vr(K,L) \sim  \vr(L^{\circ},K^{\circ})$.
\item If  $T :X_L \to X_K$ is a linear operator we have that $\frac{1}{\norm{T: X_L \to X_K}} \cdot T(L) \subset K$ and so
\begin{align*}
\vr(K,L) &\leq  \frac{\norm{T: X_L \to X_K}\vn{K}}{\vert \det{T} \vert^{\frac{1}{n}}\vn{L}}.
\end{align*}
\item $\vr(K,L) \leq \vr(K,Z) \cdot \vr(Z,L)$ for every convex body $Z$ in $\R^n$.
\item $\vr(K,L)=\vr(T(K),S(L))$, for any affine transformations $T$ and $S$. In other words, the volume ratio between $K$ and $L$ depends exclusively on the affine classes of the bodies involved.
\end{enumerate}
\end{remark}

Notice that by Rogers-Shephards inequality, for every convex body $L \subset \R^n$ we have $\vr(L-L,L) \leq 4$. Therefore, by the last property
$$\vr(K,L) \leq \vr(K,L-L) \cdot 4.$$
Thus, the largest volume ratio of the body $K$ can be estimated by considering the supremum over all symmetric bodies. Precisely,
\begin{align}\label{bastasimetricos}
\lvr(K) \leq 4 \sup_{L \subset \R^n} \vr(K,L),
\end{align}
where the $\sup$ runs over all the \emph{centrally symmetric convex   bodies} $L$.
This will be useful since it allow us to deal only with  bodies which are centrally symmetric.
\section{Lower bound for the largest volume ratio}

We now treat lower bounds for the largest volume ratio of a given convex body $K$.
Recall the statement of Remark \ref{propiedades elementales} (1),

$$\vr(K,L) = \left(\frac{|K|}{|L|}\right)^{\frac{1}{n}} \cdot \inf_{T \in SL(n,\R)} \Vert T: X_L \to X_K \Vert.$$

Therefore, to show ``good'' lower bounds for $\lvr(K)$ we need a body $L$ such that its volume is ``small'' and the norm $\Vert T: X_L \to X_K \Vert$ is large for every operator $T \in SL(n,\R)$.

The key idea of \cite{khrabrov2001generalized} is to use the probabilistic method.
Namely, Khrabrov considered
the random body (based on Gluskin's work \cite{gluskin1981diameter})
\begin{align}\label{def random politope}
L^{(m)}:=\absconv\{X_1, \dots, X_m, e_1, \dots, e_n \},
\end{align}
where $\{X_i\}_{i=1}^m$ are independent vectors distributed according to the normalized Haar measure in $S^{n-1}$.

\begin{figure}[h]
\begin{center}
\begin{tikzpicture}
\definecolor{greet}{RGB}{50,0,50}
\begin{scope}[scale=2]
 \draw (0,0) circle (1cm);
\def\puntos{(1,0),(0,-1),(-1,0),(0,1)} 
\def\us{1/0,0.59/-0.81,0/-1,-0.81/-0.58,-1/0,-0.59/0.81,0/1,0.81/0.58}
\foreach \x/\y [count=\i from 1] in \us {
   \coordinate (u\i) at (\x,\y);
 } 
\draw (u1)--(u2)--(u3)-- (u4)--(u5)--(u6)--(u7)--(u8)--cycle;
\fill[greet,opacity=0.3] (u1)--(u2)--(u3)-- (u4)--(u5)--(u6)--(u7)--(u8)--cycle;
\node at (u1) [right = 1 mm] {$e_1$}; 
\node at (u2) [below = 2.5 mm,right = 0.5  mm] {$X_1$};
\node at (u3) [below = 1 mm] {$-e_2$};
\node at (u4) [left = 4.5 mm, below= 0.5 mm] {$-X_2$};
\node at (u5) [left = 1 mm] {$-e_1$};
\node at (u6) [above = 3 mm,left=0.1 mm] {$-X_1$};
\node at (u8) [above = 2 mm,right = 1 mm] {$X_2$};
\node at (u7) [above = 1 mm] {$e_2$}; 
 \end{scope}

\end{tikzpicture}
\caption{Random polytope $L^{(2)}$ in $\R^2$.}
\end{center}
\end{figure}
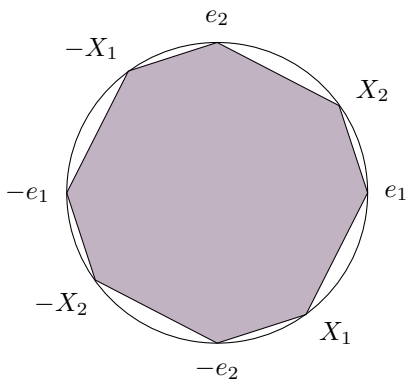

Note that as $m$ grows,   $\inf_{T \in SL(n,\R)} \Vert T: X_{L^{(m)}} \to X_K \Vert$  becomes larger but  $\frac{1}{\vert L^{(m)} \vert^{1/n}}$ decreases, so there is some sort of trade-off.

It should be noted that the volume of the random polytope $L^{(m)}$ is bounded by (see \cite{barany1988approximation})
\begin{align}\label{volumen politopo}
    \vn{L^{(m)}} \ll  \frac{\sqrt{\log(\frac{m}{n})}}{n}.
\end{align}
In fact, this bound is the exact asymptotic growth of $\vn{L^{(m)}}$  with probability greater than or equal to $1- \frac{1}{m}$ \cite[Chapter 11]{brazitikos2014geometry}.

 In \cite{khrabrov2001generalized}, for $m=n \log(n)$, it is shown that, with high probability, the norm $\norm{T: X_{L^{(m)}} \to X_K}$ is ``large'' for every $T \in SL(n,\R)$.
To achieve all this he proved the following interesting inequality:
\bigskip

{\sl If $K\subset \R^n$ is in L\"owner position then for very $m \in \N$ and every $\beta>0$,
\begin{align}\label{desig khrabrox}
\mathbb P  &\left\{ \mbox{There exists }   T\in SL(n,\R)  :  \norm{T: X_{L ^{(m)}}\to X_K} \leq  \beta \left(\frac{\vert B_2^n \vert}{\vert K \vert} \right)^{1/n}  \right\}
\\
 \nonumber & \leq \left(C \sqrt{n}  \right)^{n^2} \left(\frac{\vert B_2^n \vert}{\vert K \vert} \right)^{n} \beta^{nm-n^2}.
\end{align}
}

In order to prove our main contribution, Theorem \ref{teo principal}, we present the following refinement of the previous estimate.

\begin{proposition}\label{main prop}
Let $K\subset \R^n$ be a centrally symmetric convex body and  $L^{(m)}$ the random polytope defined in \eqref{def random politope}, then for every $\beta>0$ we have
\begin{align*}
\mathbb P & \left\{ \mbox{There exists }  T\in SL(n,\R)  :  \norm{T: X_{L^{(m)}} \to X_K} \leq  \beta \left(\frac{\vert B_2^n \vert}{\vert K \vert} \right)^{1/n}  \right\}
\\
 & \leq C^{n^2} \left(\norm{id:\ell_2^n \to X_K}\sqrt{n}|K|^{\frac{1}{n}}\right)^{n^2} (2\beta)^{nm}.
\end{align*}
\end{proposition}

To prove Proposition~\ref{main prop} we need a couple of lemmas.
The first one is a technical tool which bounds the number of points in an $\varepsilon$-net for an adequate set. This should be compared with  \cite[Lemma 5]{khrabrov2001generalized}: note that the set and the metric differ. This subtle but important modification is the key ingredient we need.
\begin{lemma}\label{lemarednuevo}
Let $K\subset \R^n$ be a convex body, $\gamma >0$ and
\begin{align*}
\MM^K_{\gamma}:=\left\{T \in SL(n,\R) \; and \;  \norm{T: \ell_1^n \to X_K} \leq \gamma\right\}.
\end{align*}
There is a $\gamma$-net, $\NN_{\gamma}^K$ for $\MM^K$ in the metric $\LL(\ell_2^n,X_K)$ such that
\begin{align*}
\#\NN_{\gamma}^K \leq C^{n^2}  \left(\norm{id:\ell_2^n \to X_K}\sqrt{n} |K|^{1/n}\right)^{n^2}.
\end{align*}
\end{lemma}
\begin{proof}
Let $U$ be the unit ball of $\LL(\ell_2^n,X_K)$. By the standard identification we consider $\MM^K_{\gamma}$ and $U$ as subsets of $\R^{n\times n}$. Let $\NN_{\gamma}^K$ be a maximal collection of elements of $\MM^K_{\gamma}$ $\gamma$-separated. These elements form an $\gamma$-net and, for every $\xi \in \NN_{\gamma}^K$, the balls $\xi + \frac{\gamma}{2}U$ are disjoints. Since
\begin{align*}
    \norm{T: \ell_1^n \to X_K} \leq  \norm{T: \ell_2^n \to X_K},
\end{align*}
we have that $\gamma U \subset \left\{T:  \norm{T: \ell_1^n \to X_K} \leq \gamma\right\}$ and then
\begin{align*}
    \bigcup_{\xi \in \NN_{\gamma}^K}\xi + \frac{\gamma}{2} U \subset \frac{3}{2}\left\{T:  \norm{T: \ell_1^n \to X_K} \leq \gamma\right\}.
\end{align*}
Computing the volume on both sides, we get the following bound for $\#\NN_{\varepsilon}^K$,
\begin{align}
    \#\NN_{\gamma}^K \left(\frac{\gamma}{2}\right)^{n^2}| U| \leq \left(\frac{3}{2}\right)^{n^2}\left| \left\{T:  \norm{T: \ell_1^n \to X_K} \leq \gamma\right\}\right| \nonumber\\\label{cota red fraccion}
    \#\NN_{\gamma}^K \leq \left(\frac{3}{\gamma}\right)^{n^2}\frac{\left| \left\{T:  \norm{T: \ell_1^n \to X_K} \leq \gamma\right\}\right|}{|U|}.
\end{align}
Now notice that
\begin{align}
\left\{T \in \mathcal L(\ell_1^n,X_K) :  \norm{T: \ell_1^n \to X_K} \leq \gamma\right\} \nonumber\\  \subset \left\{X \in \R^{n\times n} : X_i \in \gamma\cdot K \;\mbox{for all } i \right\} \nonumber\\
\subset  \underbrace{\left(\gamma K\right) \times \dots \times \left(\gamma K\right)}_{n},
\end{align}
and hence
\begin{align}\label{cota MK}
    \left|\left\{T:  \norm{T: \ell_1^n \to X_K} \leq \gamma \right\}\right| \leq (\gamma)^{n^2}|K|^n.
\end{align}
In order to bound Equation~\eqref{cota red fraccion} we need a lower bound for $|U|$. By passing to spherical coordinates it can be checked that
\begin{align}
\frac{|U|}{|B_2^{n^2}|}  = \int\limits_{S^{n^2-1}}\norm{T}_{\LL(\ell_2^n,X_K)}^{-n^2}d\sigma(T),	
\end{align}
where $\sigma$ is the normalized Haar measure on $S^{n^2-1}$. Now we apply Hölder's inequality to get
\begin{align*}
1 &\leq \left(\int\limits_{S^{n^2-1}}\norm{T}_{\LL(\ell_2^n,X_K)}^2d\sigma(T)\right)^{\nicefrac{1}{2}}\left(\int\limits_{S^{n^2-1}}\norm{T}_{\LL(\ell_2^n,X_K)}^{-2}d\sigma(T)\right)^{\nicefrac{1}{2}}\\
&\leq  \left(\int\limits_{S^{n^2-1}}\norm{T}_{\LL(\ell_2^n,X_K)}^2d\sigma(T)\right)^{\nicefrac{1}{2}}\left(\int\limits_{S^{n^2-1}}\norm{T}_{\LL(\ell_2^n,X_K)}^{-n^2}d\sigma(T)\right)^{\nicefrac{1}{n^2}}.
\end{align*}
Therefore,
\begin{align*}
\frac{|U|}{|B_2^{n^2}|} \geq \left(\int\limits_{S^{n^2-1}}\norm{T}_{\LL(\ell_2^n,X_K)}^2d\sigma(T)\right)^{\nicefrac{-n^2}{2}}.
\end{align*}
By comparing spherical and Gaussian means  and applying Gaussian Chevet's inequality \cite[Equations (12.7),(43.1)]{tomczak1989banach}, we have that
\begin{align*}
\left(\int\limits_{S^{n^2-1}}\norm{T}_{\LL(\ell_2^n,X_K)}^2d\sigma(T)\right)^{\nicefrac{1}{2}} \ll \frac{1}{n} \left(\ell(K) + \norm{id:\ell_2^n \to X_K} \sqrt{n}\right),
\end{align*}
which implies
\begin{align*}
    \left(\ell(K) + \norm{id:\ell_2^n \to X_K} \sqrt{n}\right)^{-n^2}C^{-n^2} \geq |U|.
\end{align*}
Now notice that since $B_2^n \subset \norm{id:\ell_2^n \to X_K} K$ we have that, 
\begin{align*}
\frac{1}{\norm{id:\ell_2^n \to X_K}}K^{\circ} \subset  B_2^n
\end{align*}
 and hence $\omega(K^{\circ}) \leq \norm{id:\ell_2^n \to X_K}$. Recalling that $\ell(K) \sim \sqrt{n} \omega (K^{\circ})$ we get,
 \begin{align}
 \ell(K) \leq \norm{id:\ell_2^n \to X_K} \sqrt{n}.
 \end{align}
Thus, 
\begin{align}\label{cota U}
    \left(\norm{id:\ell_2^n \to X_K} \sqrt{n}\right)^{-n^2}C^{-n^2} \geq |U|.
\end{align}

Using  Equations \eqref{cota MK} and \eqref{cota U} in Equation~\eqref{cota red fraccion} we get the desired bound.
\end{proof}
We also need the following result.
\begin{lemma}[\cite{tomczak1989banach}, Lemma 38.3]\label{lemavol}
Let $K \subset \R^n$ be a convex body, $L^{(m)}$ the random polytope in \eqref{def random politope}, $T \in SL(n,\R)$ and $\alpha > 0$. Then
\begin{align}
\Pro \left\{ \norm{T: X_{L^{(m)}} \to X_K} \leq \alpha \right\} \leq \alpha^{mn}\left(\frac{|K|}{|B_2^n|}\right)^m.
\end{align}
\end{lemma}

We present the proof of Proposition~\ref{main prop}.

\begin{proof}[Proof of Proposition~\ref{main prop}]
Let $\{X_i\}_{i=1}^m \subset \Sn$ and $L^{(m)}$ be the polytope in \eqref{def random politope} such that there exists $T\in SL(n,\R)$ with $\norm{T: X_{L^{(m)}} \to X_K} \leq  \gamma $.
As $\ell_1^n \subset L^{(m)}$, $T$ lies in the set $\MM^K$ defined in Lemma~\ref{lemarednuevo}. Consider a $\gamma$-net, $\NN_{\gamma}^K$ for $\MM^K$ for the metric $\LL(\ell_2^n,X_K)$ such that
\begin{align}\label{cardinal}
\#\NN_{\gamma}^K \leq C^{n^2}  \left(\norm{id:\ell_2^n \to X_K}\sqrt{n}\vn{K}\right)^{n^2}.
\end{align}
Let $S \in \NN_{\gamma}^K$ such that $\norm{S-T}_{\LL(\ell_2^n,X_K)} \leq \gamma$, then

\begin{align*}
\norm{S: X_{L^{(m)}} \to X_K} &\leq \norm{T:X_{L^{(m)}} \to X_K} + \norm{S-T:X_{L^{(m)}} \to X_K} \\
&\leq \gamma +\norm{S-T: \ell_2^n \to X_K} \\
&\leq 2 \gamma,
\end{align*}
where we have used the fact that $\norm{S-T:X_{L^{(m)}} \to X_K}\leq \norm{S-T: \ell_2^n \to X_K}$ since by construction $L^{(m)}\subset B_2^n$.
Hence,
\begin{align*}
\BB_{\gamma}:=\left\{ \mbox{There exists }  T\in SL(n,\R)  :  \norm{T: X_{L^{(m)}} \to X_K} \leq  \gamma \right\} \\ \subset \bigcup_{S \in \NN_{\gamma}^K}\left\{\norm{S: X_{L^{(m)}} \to X_K} \leq 2\gamma \right\}.
\end{align*}
Take $\gamma_0 := \beta\left(\frac{|B_2^n|}{|K|}\right)^{\frac{1}{n}}$, by the union bound, Equation~\eqref{cardinal} and Lemma~\ref{lemavol}
\begin{align}
\P(\BB_{\gamma_0}) \leq C^{n^2} \left( \norm{id:\ell_2^n \to X_K}\sqrt{n}|K|^{\frac{1}{n}}\right)^{n^2} (2\beta)^{nm},
\end{align}
which concludes the proof.
\end{proof}
As a consequence of Proposition \ref{main prop} we obtain the following result.
\begin{proposition}\label{teo: mejora khr}
Let $K\subset \R^n$ be a centrally symmetric convex body such that
\begin{align*}
    \norm{id:\ell_2^n \to X_K}\sqrt{n}|K|^{\frac{1}{n}} \sim 1.
\end{align*}
Given $\delta \ge 1$, with probability greater than or equal to $1-e^{-n^2}$ the random polytope $L^{(\lceil \delta n \rceil)}$ in \eqref{def random politope} verifies
$$\sqrt{n} \ll \vr(K,L^{(\lceil \delta n \rceil)}).$$
In particular, $\sqrt{n} \ll \lvr(K)$.
\end{proposition}
\begin{proof}
By Proposition~\ref{main prop} we know that there is an absolute constant $C>0$ such that, for every $\beta>0$,
\begin{align*}
\mathbb P & \left\{ \mbox{There exists }  T\in SL(n,\R)  :  \norm{T: X_{L^{(m)}} \to X_K} \leq  \beta \left(\frac{\vert B_2^n \vert}{\vert K \vert} \right)^{1/n}  \right\}
\\ &\leq C^{n^2}(2\beta)^{nm}.
\end{align*}
If $m=\lceil \delta n \rceil$ and $\beta  \leq  \frac{1}{2}(Ce)^{-\frac{1}{\delta}}$, then with probability at least $1-e^{-n^2}$ the random polytope verifies
\begin{align} \label{norma op}
    \norm{T: X_{L^{(\lceil \delta n \rceil)}} \to X_K} \geq  \beta \left(\frac{\vert B_2^n \vert}{\vert K \vert} \right)^{1/n} \sim \frac{1}{\sqrt{n} \vn{K}},
\end{align}
for every $T \in SL(n,\R)$.

Hence, by Equations \eqref{volumen politopo} and \eqref{norma op} and Remark \ref{propiedades elementales} (1) we have
\begin{align*}
  \sqrt{n} \ll \vr(K,L^{(\lceil \delta n \rceil)}),
\end{align*}
which concludes the proof.
\end{proof}
In order to prove Theorem \ref{main prop} we will show that any given convex body can be approximated by another one which fulfils the hypothesis of the previous proposition. 
 To do this we use Klartag’s solution to the isomorphic slicing problem (which asserts that, given any convex body, we can find another convex body, with absolutely bounded
isotropic constant, that is geometrically close to the first one).

\begin{theorem}[\cite{klartag2006convex}, Theorem 1.1 ]\label{teo aprox L acotado}
Let $K \subset \R^n$ be a convex body and let $\varepsilon > 0$. Then there is a convex body $T \subset \R^n$ such that
\begin{enumerate}
 \item $d(K,T) < 1 + \varepsilon$,
 \item $L_T < \frac{c}{\sqrt{\varepsilon}}$.
\end{enumerate}
Here $c>0$ is an absolute constant and $$d(K,T)=\inf \{ab : a,b >0, \exists \; x, y \in \R^n, \frac{1}{a}( K + x) \subset T + y \subset b (K +x)\}.$$
\end{theorem}
\begin{remark}\label{remark aproximacion acotada}
Given a convex body $K \subset \R^n$ there is a convex body $T \subset \R^n$ such that $\vr(T,K) \sim \vr (K,T) \sim 1$ and $L_T \leq c$, where $c>0$ is an absolute constant.
\end{remark}

Indeed, given $K$, by Theorem  \ref{teo aprox L acotado} (using $\varepsilon = 1$)  there is $T \subset \R^n$ with $L_T \leq c$ and $d(K,T) \leq 2$. Notice that if for certain $x,y \in \R^n$ and $a,b >0$ we have that $\frac{1}{a}( K + x) \subset T + y \subset b (K +x)$. Then,
\begin{align*}
\vr(T,K) \leq \frac{\vn{T}}{\frac{1}{a}\vn{K}} \leq ab \frac{\vn{K}}{\vn{K}} \leq ab.
\end{align*}
Hence $vr(T,K) \leq d(T,K)$, and by symmetry, the same holds for $\vr(K,T)$.

\begin{proposition}\label{propo aprox}
For every convex body $K\subset \R^n$ there is a convex body $W$ with $\vr(W,K) \sim 1$ such that
\begin{align}
    \norm{id:\ell_2^n \to X_W}\sqrt{n}|W|^{\frac{1}{n}} \sim 1.
\end{align}
\end{proposition}

\begin{proof}[Proof of Proposition \ref{propo aprox}]
By Remark \ref{remark aproximacion acotada} and the Roger-Shephard inequality (replacing the body if necessary) we can assume that $K^{\circ}$ is a centrally symmetric isotropic convex body and  $L_{K^{\circ}}$ is uniformly bounded.

By Markov's inequality we have that
\begin{align*}
    \P\{\norm{x}_2 \geq c\sqrt{n} \} \leq \frac{1}{2}
\end{align*}
for some absolute constant $c>0$.

Consider $W$ such that $W^{\circ} = K^{\circ}\cap  c\sqrt{n}B_2^n$, then  $\vn{W^{\circ}} \geq \frac{1}{2}$ and hence $\vr(W,K) \sim \vr(K^{\circ},W^{\circ}) \sim 1$.

Since $W^{\circ} \subset c \sqrt{n}B_2^n$ we have that  $$\norm{id:\ell_2^n \to X_W}=\norm{id: X_{W^{\circ}} \to \ell_2^n} \ll  \sqrt{n}.$$
Finally, as $\vn{W^{\circ}} \sim 1$, we have that $\vn{W} \sim \frac{1}{n}$ (applying the Blaschke-Santaló / Bourgain-Milman inequality,  Equation~\eqref{santalo}). Therefore 
\begin{align*}
     \norm{id:\ell_2^n \to X_W}\sqrt{n}|W|^{\frac{1}{n}} \sim 1,
    \end{align*}
    which concludes the proof.
\end{proof}

As pointed out by A. Giannopoulos and  A. Litvak \cite{GL}, an alternative proof of the last proposition  can be obtained by using the $M$-position, which seems a more natural approach.

The $M$-position was discovered by Milman in relation with the reverse Brunn-Minkowski inequality. We refer the reader to \cite[Chapter 8]{artstein2015asymptotic} for a very nice treatment on this topic. 

A convex body $K$ is in $M$-position with constant $C$ if letting $rB_2^n$ have the same volume as $K$, i.e. $r=\vrn{K}{B_2^n}$, we have
\begin{align}\label{mposicion}
    \frac{1}{C} \vn{rB_2^n + T} \leq \vn{K + T} \leq C\vn{rB_2^n + T},\\
    \nonumber \frac{1}{C} \vn{r^{-1}B_2^n + T} \leq \vn{K^\circ + T} \leq C\vn{r^{-1}B_2^n + T},
\end{align}
for every convex body $T \subset \R^n$.
Milman proved that there is an absolute constant $C>0$ such that for every centrally symmetric convex body $K$ there is an transformation $A \in SL(n,\R)$ such that $AK$ is in $M$-position with constant $C$.

\begin{proof}[Alternative proof of Proposition \ref{propo aprox}]
We can suppose that $K$ is a centrally symmetric body in $M$-position. By Equation \eqref{mposicion} the convex body 
$$W := K + \vrn{K}{B_2^n}B_2^n$$ satisfies that $\vn{W} \sim \vn {K}$ and therefore $\vr(W,K) \sim 1$.
On the other hand, since $\vrn{K}{B_2^n}B_2^n \subset W$, we have that $\norm{id: B_2^n \to W} \leq \vrn{B_2^n}{K}$.

Hence, \begin{align*}
    \norm{id:\ell_2^n \to X_W}\sqrt{n}|W|^{\frac{1}{n}} \sim 1,
\end{align*}
as wanted.
\end{proof}

The following theorem contains, as a consequence, Theorem \ref{teo principal}. 
\begin{theorem}\label{teo principal proba}
Let $K\subset \R^n$ be  convex body.
Given $\delta \ge 1$, with probability greater than or equal to $1-e^{-n^2}$ the random polytope $L^{(\lceil \delta n \rceil)}$ in \eqref{def random politope} verifies
$$\sqrt{n} \ll \vr(K,L^{(\lceil \delta n \rceil)}).$$
In particular, $\sqrt{n} \ll \lvr(K)$.
\end{theorem}
\begin{proof}
By Proposition \ref{propo aprox} there is $W$ with $\vr(W,K) \sim 1$ such that
\begin{align}
     \norm{id:\ell_2^n \to X_W}\sqrt{n}|W|^{\frac{1}{n}} \sim 1.
\end{align}
Applying Proposition \ref{teo: mejora khr}, given $\delta \ge 1$, with probability greater than or equal to $1-e^{-n^2}$ the random polytope $L^{(\lceil \delta n \rceil)}$ in \eqref{def random politope} verifies
$$\sqrt{n} \ll \vr(W,L^{(\lceil \delta n \rceil)}).$$
Then,
\begin{align*}
\sqrt{n} \ll \vr(W,L^{(\lceil \delta n \rceil)}) \leq \vr(W,K)\vr(K,L^{(\lceil \delta n \rceil)}) \sim \vr(K,L^{(\lceil \delta n \rceil)}),
\end{align*}
as wanted.
\end{proof}

\begin{corollary}
Let $K\subset \R^n$ be a convex body.
Given $\delta \ge 1$, there is polytope $Z^{(\lceil \delta n \rceil)}$ with $2(\lceil \delta n \rceil +n)$ facets such that
$$\sqrt{n} \ll \vr(Z^{(\lceil \delta n \rceil)},K).$$
\end{corollary}

\begin{proof}
The result follows from the previous theorem, Remark \ref{propiedades elementales} (2) and the fact that the polar of the polytope $L^{(\lceil \delta n \rceil)}$ has $2(\lceil \delta n \rceil +n)$ facets.
\end{proof}
\section{Upper bounds} 
We now provide upper estimates for $\lvr(K)$ for different classes of convex bodies. Together with this inequalities we derive sharp asymptotic estimates. 
It should be mentioned that, to bound $\vr(K,L)$,  Giannopoulos and Hartzoulaki  \cite{giannopoulos2001john}  managed to find randomly a unitary operator $T:X_K \to X_L$ with small norm. To do this, they used Chevet's inequality for an adequate position of $L$.
To our purposes we will use the following high probability version of the Gaussian Chevet's inequality (tail inequality).
\begin{proposition}\label{highprobChevet}
Let $A = \left(g_{ij}\right)_{1\leq i,j \leq n} \in \R^{n\times n}$ be a random matrix with independent gaussian entries $g_{ij}\sim \NN(0,1)$ and $K,L \subset \R^n$ two convex bodies. Then, for all $u \geq 0$, with probability greater than $1 - e^{-u^2}$ we have
\begin{align}
\norm{A:X_L \to X_K} \ll &\ell(K) \norm{id: \ell_2^n \to X_{L^{\circ}}} + \ell(L^{\circ}) \norm{id:\ell_2^n\to X_{K}} \\ \nonumber &+ u\norm{id: \ell_2^n \to X_{L^{\circ}}} \cdot  \norm{id: \ell_2^n \to X_{K}}
\end{align}
\end{proposition}

Although the previous proposition is probably known for specialist we were not able to find an explicit reference of it (the closest statement we found is \cite[Exercise 8.7.3]{vershynin2018high}). We include a sketch of its proof for completeness.

\begin{proof}[Sketch of the proof of Proposition~\ref{highprobChevet}]
We define in $L \times K^{\circ}$ the distance
\begin{align*}
d((x,y^*),(\tilde{x},\tilde{y}^*)):= \norm{x - \tilde{x}}_2 \norm{id: \ell_2^n \to X_K} + \norm{y^* - \tilde{y^*}}_2 \norm{id:\ell_2^n \to X_{L^{\circ}}},
\end{align*}
where $\Vert \cdot \Vert_2$ stands for the Euclidean norm.

Consider the random process in $L \times K^{\circ}$ given by $$X_{(x,y^*)} := \langle Ax, y^* \rangle.$$ It is not hard to see that this process is subgaussian for $d$ (see the proof of \cite[Theorem 8.7.1]{vershynin2018high}); i.e., $$\norm{X_{(x,y^*)} - X_{(\tilde{x},\tilde{y}^*)}}_{\psi_2} \leq Cd((x,y^*),(\tilde{x},\tilde{y^*})).$$
Note that if we consider the Gaussian process

$$Y_{(x,y^*)}:=  \langle g, x \rangle \norm{id: \ell_2^n \to X_K} + \langle h, y^* \rangle \norm{id:\ell_2^n \to X_{L^{\circ}}},$$
where $g=(g_1, \dots, g_n)$, $h=(h_1, \dots, h_n)$ and $(g_i)_{i=1}^n, (h_j)_{j=1}^n$ are independent standard Gaussian variables; we have
$$
\Vert Y_{(x,y^*)} - Y_{(\tilde x,\tilde y^*)} \Vert_2 = d((x,y^*),(\tilde{x},\tilde{y^*})).
$$

Combining the generic chaining (tail bound) \cite[Theorem 8.5.5]{vershynin2018high} and Talagrand’s majorizing measure theorem \cite[Theorem 8.6.1]{vershynin2018high} we get
\begin{align*}
\norm{A:X_L \to X_K} =& \sup_{(x,y^*) \in L \times K^{\circ}}X_{(x,y^*)} \\&\ll \left( \E [\sup_{(x,y^*) \in L \times K^{\circ}}Y_{(x,y^*)}] + u \; \diam(K \times L^{\circ}) \right),
\end{align*}
with probability at least $1-e^{-u^2}$.

The result follows by the fact that
$$\E [\sup_{(x,y^*) \in L \times K^{\circ}}Y_{(x,y^*)}]= \ell(K) \norm{id: \ell_2^n \to X_{L^{\circ}}} + \ell(L^{\circ}) \norm{id:\ell_2^n\to X_{K}}$$ and
$\diam(L \times K^{\circ})\sim \norm{id: \ell_2^n \to X_{L^{\circ}}} \cdot  \norm{id: \ell_2^n \to X_{K}}$.
\end{proof}

Given a convex body $W \subset \R^n$ we need to introduce a position $\tilde W$ highly related with the well-known $\ell$-position.
It has been introduced by Rudelson in \cite{rudelson2000distances} and its existence can be also tracked in the proof of the main theorem of the paper of Giannopoulos and Hartzoulaki   \cite{giannopoulos2002volume}, who used this position together with Chevet's inequality to bound the volume ratio. 
   
\begin{proposition} 
Given a convex body $W \subset \R^n$ there is position of $W$,$\tilde{W}$ that satisfies:
\begin{itemize}
\item $\ell(\tilde{W}) \lss \sqrt{n} \log(n)$,
\item $\ell(\tilde{W}^{\circ}) \lss \sqrt{n}$,
\item $\norm{id:\ell_2^n \to X_{\tilde{W}^{\circ}}} \lss \frac{\sqrt{n}}{\log(n)}$.

\end{itemize}
In particular,  $$\frac{1}{|\tilde{W}|^{\frac{1}{n}}} \leq \ell(\tilde W) \lss \sqrt{n}\log(n).$$

\end{proposition}
When a convex body in $\R^n$ satisfies the previous estimates we say it is in \emph{Rudelson position}.

First we will show how Rudelson's position together with Chevet's inequality can be used to bound the largest volume ratio for some natural classes of convex bodies.
Observe that, by Remark \ref{propiedades elementales} (3), bounding simultaneously the determinant (from bellow) and the norm (from above) of an operator gives a bound for the volume ratio.
We will also need the following lower bound for the determinant of a random Gaussian matrix, which can be found in  \cite[Corollary 1]{pivovarov2010determinants}.

\begin{lemma}
Let $A = \left(g_{ij}\right)_{1\leq i,j \leq n} \in \R^{n\times n}$ with $g_{ij} \sim \NN(0,1)$, then with probability at least $1-e^{-n}$ we have
\begin{align}\label{determinante gaussiana}
\det(A)^{\frac{1}{n}} \gss \sqrt{n}.
\end{align}
\end{lemma}
Combining the last inequality together with Proposition~\ref{high prob upper bound}, we can ensure that for any $u \leq \sqrt{n}$, with probability greater than $1 - 2e^{-u^2}$, a random Gaussian operator $A$ fulfils both,
$\det(A)^\frac{1}{n} \gss \sqrt{n}$,
and
\begin{align}\label{eq: norma chevet}
\norm{A:X_L \to X_K} \lss &\ell(K) \norm{id: \ell_2^n \to X_{L^{\circ}}} + \ell(L^{\circ}) \norm{id:\ell_2^n\to X_{K}} \\ \nonumber &+ u\norm{id: \ell_2^n \to X_{L^{\circ}}} \cdot  \norm{id: \ell_2^n \to X_{K}},
\end{align}
for any pair of convex bodies $K,L \subset \R^n$.

Assume that $L$ is in Rudelson's position. Now combining equations \eqref{eq: norma chevet} and Remark \ref{propiedades elementales} (3),
we have that $\frac{TL}{\norm{T}}  \subset K$ and
\begin{align} \label{eq: cota combinada}
\vrn{K}{\frac{TL}{\norm{T}}}  \lss \ell(K) \vn{K} \sqrt{n} + \log(n) \sqrt{n} \norm{id: \ell_2^n \to X_K}\\
+ u \sqrt{n} \norm{id: \ell_2^n \to X_K} \vn{K}, 
\end{align}
with probability greater on equal to $1 - 2e^{-u^2}$.

In the following we are going to use the previous relations to bound the volume ratio for some classes of convex body. 

\subsection{Unitary invariant norms}

Recall that unitary invariant norm $\NN$ on $\R^{d \times d}$, that is a norm that satisfies $\NN(UTV) = \NN(T)$ for all $U,V \in \OO(d)$ (the group of distance-preserving linear transformations of a Euclidean space of dimension $d$). 
The norm $\sigma_p$ is one of the most important unitary invariant operator norms.
It is known that for any unitary invariant norm $\NN$ there is a $1-$symmetric norm $\tau$ such that for every $T \in \R^{d \times d}$
\begin{align*}
\NN(T) = \tau (s_1(T),\dots,s_n(T)).
\end{align*}
Assume that $\tau(e_i) = 1$ and set $u:= \sum_{i=1}^d e_i$. By \cite[Equation (4.3.6)]{brazitikos2014geometry}
\begin{align}\label{relacion Schatten}
\frac{1}{\tau(u)}\SS_{\infty}^d \subset B_{\NN} \subset \frac{d}{\tau(u)}\SS_1^d.
\end{align}
Taking volumes we have that
\begin{align*}
\frac{1}{\tau(u)}|\SS_{\infty}^d |^{\frac{1}{d^2}} \leq |B_{\NN}|^{\frac{1}{d^2}} \leq  \frac{d}{\tau(u)}|\SS_1^d|^{\frac{1}{d^2}}.
\end{align*}
As $|\SS_{\infty}^d |^{\frac{1}{d^2}} \sim d|\SS_1^d|^{\frac{1}{d^2}}$ \cite[Lemma 4.3.2]{brazitikos2014geometry}, we conclude that
\begin{align} \label{vr schatten s1 y sinfty}
    \vr(B_{\NN},\SS_{\infty}^d) \sim \vr(\SS_1^d,B_{\NN}) \sim 1.
\end{align}

\begin{theorem} \label{high prob upper bound}
Let $B_{\NN}$ be the unit ball of any unitary invariant norm $\NN$ in $\R^{d \times d}$ and $L \subset \R^{d^2}$ a convex body in Rudelson position, and let $A = \left(g_{ij}\right)_{1\leq i,j \leq d^2} \in \R^{d^2 \times d^2}$ be a random matrix with independent Gaussian entries $g_{ij}\sim\NN(0,1)$. Then with probability greater than $1 -2 e^{-d}$, the body $\tilde{L}:= \frac{AL}{\norm{A}} \frac{1}{\tau(u)} \subset B_{\NN}$ and also $$\frac{\vd{B_{\NN}}}{\vd{\tilde{L}}} \ll d.$$
\end{theorem}

As a consequence of Theorem \ref{teo principal}, the previous result and Remark \ref{propiedades elementales} (5) we obtain the following corollary.
\begin{corollary} \label{cota superior unitary}
Let $B_{\NN}$ be the unit ball of any unitary invariant norm $\NN$ in $\R^{d \times d}$. Then,
\begin{align}
    \lvr(B_{\NN}) \sim d.
\end{align}

\end{corollary}

\begin{proof}[Proof of Theorem~\ref{high prob upper bound}]

Note that by \cite[Lemma 4.3.2]{brazitikos2014geometry} we know that $|B_{\SS^d_{\infty}}|^{\frac{1}{d^2}} \sim d^{-\frac{1}{2}}$  and also by \cite[Excercise 7.24]{aubrun2017alice} $\ell(B_{\SS^d_{\infty}}) \sim d^{\frac{1}{2}}$, hence $\ell(B_{\SS^d_{\infty}}) |B_{\SS^d_{\infty}}|^{\frac{1}{d^2}} \sim 1.$
On the other hand, since $\SS_{\infty}^d$ has enough symmetries; by Equation~\eqref{vr simetrias} and by \cite[Excercise 7.24]{aubrun2017alice} (see also \cite{kabluchko2018exact}) we know that
\begin{align*}
  \vr(B_{\SS_1^d})= \norm{id:\ell_2^{d^2} \to \SS_{\infty}^d} \cdot d \cdot  |B_{\SS_{\infty}^d}|^{\frac{1}{d^2}} \sim \sqrt{d}.
\end{align*}

Using the fact that $L$ is in Rudelson's position, by Equation~\eqref{eq: cota combinada} with $K=\SS_{\infty}^d$, $n=d^2$ and $u = \sqrt{d}$, we have that  $A(L) \subset \norm{A} \SS_{\infty}^d$, and
$$\left(\frac{\norm{A}|\SS_{\infty}^d|}{|A(L)|}\right)^{\frac{1}{d}} \leq d, $$ 
 with probability greater than $1-2e^{-d}$. 
 
 By Equation~\eqref{relacion Schatten},
\begin{align*}
\tilde{L} := \frac{1}{\lambda(\tau)}\frac{A(L)}{\norm{A}} \subset \frac{1}{\lambda(\tau)}\SS_{\infty}^d \subset B_{\NN}
\end{align*}
As $\vd{\frac{1}{\lambda(\tau)}\SS^d_{\infty}} \sim \vd{B_{\NN}}$ we obtain the desired bound.
\end{proof}

\subsection{Tensor products}
Another natural class of convex bodies for which we can obtain sharp asymptotic bounds for the largest volume ratio are the unit balls  of tensor products of $\ell_p$-spaces endowed with the projective or injective tensor norm.

In order to do prove  Theorem \ref{lvr tensores},  we need to have estimates of some geometrical parameters of the involved spaces.
Defant and Prengel \cite{defant2009volume} obtained asymptotic estimates for many of them. We summarize their results in the next proposition.

\begin{proposition}
For $m \in \N$ set $d = n^m$ and $d_s =\binom{m + n - 1}{n - 1}$.  For each  $1 \leq p \leq \infty$ we have
\begin{enumerate}
\item\label{vol1} $\left|B_{\tenes}\right|^{\frac{1}{d_s}} \sim \left|B_{\tene}\right|^{\frac{1}{d}} \sim \begin{cases} 
      n^{m(\frac{1}{2} - \frac{1}{p}) - \frac{1}{2}} & p\leq 2 \\
      n^{-\frac{1}{p}} & p\geq 2.
   \end{cases}$
\medskip   
   \item \label{vol2}$\left|B_{\tenps}\right|^{\frac{1}{d_s}} \sim \left|B_{\tenp}\right|^{\frac{1}{d}} \sim \begin{cases} 
      n^{1 - \frac{1}{p} - m} & p\leq 2 \\
      n^{\frac{1}{2}-m(\frac{1}{2}+\frac{1}{p})} & p\geq 2.
   \end{cases}$
   \medskip  
   \item \label{ele1}$\ell(B_{\tenes}) \sim \ell(B_{\tene}) \sim \begin{cases} 
      n^{m(\frac{1}{p} - \frac{1}{2}) + \frac{1}{2}} & p\leq 2 \\
      n^{\frac{1}{p}} & p\geq 2.
   \end{cases}$
 \medskip    
   \item\label{ele2} $\ell(B_{\tenps}) \sim \ell(B_{\tenp}) \sim \begin{cases} 
      n^{m - 1 + \frac{1}{p}} & p\leq 2 \\
      n^{m(\frac{1}{2} + \frac{1}{p}) - \frac{1}{2}} & p\geq 2.
   \end{cases}$   
   \medskip  
  \item\label{id1}$\norm{ id: \ell_2^{d_s} \to \tenes} \sim \norm{id: \ell_2^d \to \tene} \sim \begin{cases} 
		 n^{m(\frac{1}{2} - \frac{1}{p})}  & p \leq 2 \\      
     1 & p \geq  2.
   \end{cases}$
  \medskip  
  
   \item \label{id2}$\norm{ id: \ell_2^{d_s} \to \tenps} \sim \norm{id: \ell_2^d \to \tenp} \sim \begin{cases} 
	        n^{\frac{m}{2} + \frac{1}{p}-1} & p\leq 2\\
      n^{\frac{m}{p} - \frac{1}{2}} & 2 \leq  p\leq 2m \\
   
 	1  & p \geq 2m.  
   \end{cases}$
\end{enumerate}
\end{proposition}

All the proofs can be found in \cite{defant2009volume}. The comparison between the full and symmetric tensor products follows from \cite[Proposition 3.1]{defant2009volume}.
The estimates \eqref{vol1} and \eqref{vol2} are in \cite[Theorem 4.2]{defant2009volume}. For \eqref{ele1} and \eqref{ele2} see \cite[Lemma 4.3]{defant2009volume}. The proof of \eqref{id1} follows form the fact that
\begin{align*}
\norm{id: \ell_2^{n^m} \to \tene} = \norm{id: \ell_2^n \to \ell_p^n}^m.
\end{align*}
For \eqref{id2}  the result is stated in \cite[Lemma 5.2]{defant2009volume}.

In particular, for every space $X$ involved in the last proposition, we have  $$\ell(B_X)|B_X|^{\frac{1}{\dim(X)}} \sim 1.$$

Fix an arbitrary body $L$ in Rudelson position. 
If $K = B_X$, $N = \dim(X)$ and $A$ is a random Gaussian matrix, we have by equation \eqref{eq: cota combinada} 
\begin{align*}
\left(\frac{\norm{A}|K|}{|A(L)|}\right)^{\frac{1}{N}} \lss \sqrt{N} + (\log{N} + u)\sqrt{N} \norm{id: \ell_2^N \to X} |B_E|^{\frac{1}{N}}.
\end{align*}
with probability greater than $1 - 2 e^{-u^2}$.
Now, if we take for example, $X = \tene$  with $p \leq 2$, we have that $$\norm{id: \ell_2^N \to \tene} |B_{\tene}|^{\frac{1}{N}} = 
 n^{2m(\frac{1}{2} - \frac{1}{p}) - \frac{1}{2}}.$$
So, taking $u = n^{-2m(\frac{1}{2} - \frac{1}{p}) + \frac{1}{2}} \geq \log(N)$, we get
$\left(\frac{\norm{A}|K|}{|A(L)|}\right)^{\frac{1}{N}} \lss \sqrt{N}.$
It can be checked that in all cases,  $\norm{id: \ell_2^N \to X} |B_X|^{\frac{1}{N}} \lss \frac{1}{\log(N)}$. So, choosing $u^{-1} =  \norm{id: \ell_2^N \to E} |B_E|^{\frac{1}{N}}$ we have that
with high probability $\left(\frac{\norm{A}|K|}{|A(L)|}\right)^{\frac{1}{N}} \lss \sqrt{N}.$

Arguing analogously for the other cases we obtain Theorem \ref{lvr tensores}. Note that in fact we have obtained a high probability version of this theorem (a statement similar to Theorem \ref{high prob upper bound}).
\bigskip

Tensor products  can be identified  naturally with multilinear forms or homogeneous polynomials.

From now we assume that $E$ is a finite dimensional space. 
The space of bounded $m$-linear forms on $E$, endowed with the usual supremum norm,  will be denoted by $\mathcal{L}(^m E)$.
Note that the tensor $\tenee$ coincides with the space of $m$-linear operators \index{operator!$m$-linear} defined on $(E')^m$.

Recall that an operator $T: E^m \to \R$ is \emph{$m$-nuclear} \index{operator!$m$-nuclear} if can be written as
\begin{align*}
T = \sum_{i=1}^{\infty} \varphi^i_1 \dots \varphi^i_m,
\end{align*}
with $\varphi \in E'$ and $\sum_{i=1}^{\infty} \norm{\varphi^i_1}_{E'} \dots \norm{\varphi^i_m}_{E'} < \infty$. We can define the following norm on the space of all $m$-nuclear operators
\begin{align*}
\norm{T}_{nuc} = \inf\{\sum_{i=1}^{\infty} \norm{\varphi^i_1}_{E'} \dots \norm{\varphi^i_m}_{E'}\},
\end{align*}
where infimum is taken over all representation of $T$ as above.  We denote this space by $\mathcal{L}_{nuc}(^m E)$. 

The space of all $m$-nuclear operators on $(E')^n$, $\mathcal{L}_{nuc}(^m E')$, can be identified with $\tenpe$.

The tensor products $\tenese$ and $\tenpse$ can be represented as spaces of polynomials.
Recall that a function $p: X \to \R$ is said to be an $m$-homogeneous polynomial if there is an $m$-linear form $\phi : E^m \to \R$ such that $p(x) = \phi(x,\dots,x)$.  We write $\PP(^mE)$ for the set of $m$-homogeneous polynomials on $E$. If we define in $\PP(^mE')$ the norm, $\norm{p} := sup_{x \in B_E} |p(x)|,$
the space is isometric to $\tenese$.

An $m$-homogeneous polynomial is said to be nuclear if it can be written as
\begin{align*}
p(x)= \sum_{i=1} ^\infty \lambda_i (\varphi_i(x))^m, 
\end{align*}
where $\lambda_i \in \R$, $\varphi_i \in E'$ and $\sum_{i=1} ^\infty |\lambda_i|\norm{\varphi_i}_{E'} < \infty$. We write $\PP_{nuc}(^mE)$ for the space of nuclear polynomials endowed with the norm
$$\norm{p}_{nuc} = \inf\left\{ \sum_{i=1} ^\infty |\lambda_i|\norm{\varphi_i}_{E'}\right\}, $$
where the infimum is taken over all representations of $p$ as above. Note that the space $\PP_{nuc}(^mE')$ is isometrically isomorphic to $\tenpse$.

Using these identifications and Theorem \ref{lvr tensores} we have the following corollary.

\begin{corollary}
Let $1\leq p \leq \infty$ and $X$ either $\mathcal{L}(^m\ell_p^n), \mathcal{L}_{nuc}(^m\ell_p^n), \mathcal{P}(^m\ell_p^n)$  or $\mathcal{P}_{nuc}(^m\ell_p^n)$. Then we have 
\begin{align}
\lvr(B_X) \sim \sqrt{\dim(X)},
\end{align}
where $\dim(X)$ stands for the dimension of $X$ as a vector space.
\end{corollary}
\subsection{Largest volume ratio for unconditional convex bodies and random Dvoretzky-Rogers' parallelepiped}

Let $K$ be an unconditional convex body in $\R^n$ and $L$ be a centrally symmetric convex body; the following statement shows a way to find positions of $L$ (say $\tilde L$), with extremely high probability,  verifying $\tilde L \subset K$, with ratio $\left(\frac{|K|}{|\tilde L|}\right)^{\frac{1}{n}}$ bounded by $\sqrt{n}$.

\begin{theorem}\label{teo: aleatorio posicion}
Let $L \subset \R^n$ be a centrally symmetric convex body such that $L^{\circ}$ is in isotropic position and consider the random matrix
$T:=\sum_{j=1}^n X_j \otimes e_j$, where $X_1, \dots, X_n$ are independently chosen accordingly to the uniform measure in the isotropic body $L^{\circ}$. With probability greater than or equal to $1 -e^{-n}$, for every unconditional isotropic body $K \subset \R^n$, the position  $\tilde L:=\frac{1}{2 \sqrt {\pi e}}  \cdot T(L)$ lies inside $K$ and
\begin{align}
\left(\frac{|K|}{|\tilde L|}\right)^{\frac{1}{n}} \ll \frac{\sqrt{n}}{L_{L^{\circ}}}.
\end{align}
\end{theorem}

Note that as a direct consequence of Theorem \ref{teo principal}, the previous theorem and Equation~\eqref{bastasimetricos} we have
\begin{align} \label{cota superior}
    \lvr(K) \sim  \sqrt{n},
\end{align}
for every unconditional body $K \subset \R^n$ (an unconditional body is isotropic and unconditional up to a diagonal operator), which shows the upper estimates in Theorem~\ref{teo incond}.

Recall the following result of Bobkov and Nazarov \cite[Proposition 2.4 and Proposition 2.5]{bobkov2003convex} (see also \cite{lozanovskii1969some} or \cite[Proposition 4.2.4]{brazitikos2014geometry}), which asserts that the normalized $\ell_1$-ball ($\ell_\infty$-ball) in $\R^n$ is the largest set (smallest set) within the class of all
unconditional isotropic bodies (up to some universal constants).

\begin{proposition}{\cite[Proposition 2.4 and Proposition 2.5]{bobkov2003convex}} \label{contencion}
Let $K \subset \R^n$ be an unconditional isotropic convex body. Then,
\begin{align}\label{eq bobkov}
\frac{1}{2 \sqrt {\pi e}} \cdot B_{\infty}^n \subset K \subset \frac{\sqrt{6}}{2} n \cdot  B_1^n,
\end{align}
where $B_{\infty}^n$ and $B_1^n$ stand for the unit balls of $\ell_\infty^n$ and $\ell_1^n$ respectively.
\end{proposition}

It should be noted that \eqref{cota superior} can be obtained by a direct use of a classical result of Dvoretzky and Rogers. Indeed, given a centrally symmetric convex body $L \subset \R^n$, by \cite[Theorem 5A]{dvoretzky1950absolute} (see also \cite{pelczynski1991parallelepipeds}) there is a centrally symmetric parallelepiped $ P \supset L$ such that
\begin{align}\label{eq1}
\left(\frac{\vert P \vert}{ \vert L \vert} \right)^{1/n} \leq c \sqrt{n},
\end{align}
for some absolute constant $c>0$.  Thus, by Remark \ref{propiedades elementales} (5), $\vr(B_{\infty}^n,L) \ll \sqrt{n}.$
If $K$ is an unconditional body, by Proposition~\ref{contencion}  we have $\vr(K,B_{\infty}^n) \sim 1$. By Remark \ref{propiedades elementales} (4) we obtain
\begin{align}\label{eq vr incondicional}
\vr(K,L) \leq \vr(K,B_{\infty}^n) \cdot \vr(B_{\infty}^n,L) \ll  \sqrt{n}.
\end{align}

Observe that, in general, understanding how the parallelepiped $P$ in Equation~\eqref{eq1} looks like seems difficult (its construction depends on certain contact points when $L$ is in John position,  which are not easy to find explicitly), thus Theorem~\ref{teo: aleatorio posicion} seems much stronger since it provides a random algorithm that works with high probability.

We therefore state the following probabilistic construction of the Dvoretzky-Rogers' parallelepiped, which can be derived from a result of Pivovarov. Note that Theorem \ref{teo: aleatorio posicion} is a direct consequence of the next theorem together with the first inclusion of Proposition \ref{contencion}.
\begin{theorem}\label{teo: random dvoretzky}
Let $L \subset \R^n$ be a centrally symmetric convex body such that $L^{\circ}$ is in isotropic position and consider the random matrix
$T:=\sum_{j=1}^n X_j \otimes e_j$, where $X_1, \dots, X_n$ are independently chosen accordingly to the uniform measure in the isotropic body $L^{\circ}$. With probability greater than or equal to $1 -e^{-n}$, the parallelepiped $P = T^{-1}(B_{\infty}^n)$ contains $L$ and
\begin{align*}
\left(\frac{|P|}{|L|}\right)^{\frac{1}{n}} \ll\frac{\sqrt{n}}{L_{L^{\circ}}}.
\end{align*}
\end{theorem}

\begin{proof}

By \cite[Proposition
1]{pivovarov2010determinants} we know that
\begin{align} \label{det}
\Pro\left\{   \vert \det \big( \sum_{j=1}^n X_j \otimes e_j \big)\vert^{1/n} \gg \sqrt{n} L_{L^{\circ}} \right\} > 1 -  e^{-n}.
\end{align}

On the other hand since $|\langle X_i,y \rangle| \leq 1$ for all $y \in L$ and $1\leq i \leq n$ we have that $\norm{T: X_L \to \ell_{\infty}^n} \leq 1$, where $T:=\sum_{j=1}^n X_j \otimes e_j$ .

Thus, $T(L) \subset B_{\infty}^n$, or equivalently $L \subset T^{-1}(B_{\infty}^n):=P$ and the ratio
\begin{align}\label{ratio}
\left(\frac{|P|}{|L|}\right)^{\frac{1}{n}} = \frac{\vn{B_{\infty}^n}}{\vert \det{T} \vert^{\frac{1}{n}}\vn{L}}.
\end{align}

Therefore, by Equations \eqref{ratio} and \eqref{det}  and  taking into account that $\vn{L} \sim \frac{1}{n}$ (which comes by applying the Blaschke-Santaló/Bourgain-Milman inequality,  Equation~\eqref{santalo},  since $|L^{\circ}| = 1$) we have, with probability greater than or equal to $1-e^{-n}$,

\begin{align}
\left(\frac{|P|}{|L|}\right)^{\frac{1}{n}} \ll \frac{\sqrt{n}}{L_{L^{\circ}}},
\end{align}
which concludes the proof.
\end{proof}

We finish the article with a consequence of Theorem~\ref{teo: aleatorio posicion}.

\begin{corollary} \label{producto piola}
For every centrally symmetric convex  body $L \subset \R^n$ we have
\begin{align}
    \vr(B_{\infty}^n,L) \cdot L_{L^{\circ}}\ll \sqrt{n}.
\end{align}

\end{corollary}
This seems to be an improvement of the well-known inequality \cite[Proposition 3.5.13]{brazitikos2014geometry} $$L_{L} \cdot L_{L^{\circ}} \ll \sqrt{n}.$$
Indeed, by Equation~\eqref{Milman-Pajor} we known that
$$L_L \ll \vr(B_{\infty}^n,L),$$
but in general $\vr(B_{\infty}^n,L)$ can be larger than $L_L$: according to Theorem~\ref{teo principal proba} and \cite[Theorem 4.4.1]{brazitikos2014geometry} there is a polytope $L^{(2n)}$ which verifies $$\vr(B_{\infty}^n,L^{(2n)}) \gg \sqrt{n}; \;\;\mbox{and}\;\; L_{L^{(2n)}} \ll \log(n).$$

In Corollary~\ref{producto piola}, at least at first instance, one should be tempted to change $\vr(B_{\infty}^n,L)$ by $\sup\limits_{K \subset \R^n \mbox{ unc.}} \vr(K,L)$, where the infimum run all over unconditional convex bodies;  but using Proposition~\ref{contencion}, it can be seen that
$$
\vr(B_{\infty}^n,L) \sim \sup_{K \subset \R^n \mbox{ unc}} \vr(K,L).
$$

\subsection{Acknowledgement}
The authors are grateful to  Apostolos Giannopoulos and  Alexander Litvak for pointing out the alternative proof of Proposition \ref{propo aprox} which uses the $M$-position.

\end{document}